\documentclass[english]{amsart}

\usepackage{latexsym}

\usepackage{amsmath}
\usepackage{rotating}
\usepackage{amsfonts}
\usepackage{amsthm}
\usepackage{amssymb}
\usepackage{verbatim}
\usepackage{psfrag}
\usepackage{graphicx}
\newtheorem{lemm}{Lemma}[section]
\newtheorem{prop}[lemm]{Proposition}

\newtheorem{defi}[lemm]{Definition}
\newtheorem{coro}[lemm]{Corollary}
\newtheorem{rema}[lemm]{Remark}
\newtheorem{theo}{Theorem}

  \usepackage[all]{xy}
\xyoption{matrix}

\newcommand{\C}{\mathbb{C}}
\newcommand{\Bir}{\mathrm{Bir}}

\newcommand{\Aut}{\mathrm{Aut}}

\newcommand{\Sym}{\mathrm{Sym}}

\newcommand{\A}{\mathbb{A}}

\newcommand{\p}{\mathbb{P}}
\newcommand{\ZZ}{\mathbb{Z}}
\newcommand{\SL}{\mathrm{SL}}
\renewcommand{\div}{\mathrm{div}}
\newcommand{\Symp}{\mathit{Symp}}

\title[Symplectic birational transformations of the plane]{Symplectic birational transformations of the plane}
\author{J\'er\'emy Blanc}
\email{Jeremy.Blanc@unibas.ch}
\thanks{Supported by the SNSF grant no PP00P2\_128422 /1}
\begin{document}
\begin{abstract}{We study the group of symplectic birational transformations of the plane. It is proved that this group is generated by $\mathrm{SL}(2,\mathbb{Z})$, the torus and a special map of order $5$, as it was conjectured by A. Usnich.

Then we consider a special subgroup $H$, of finite type, defined over any field  which admits a surjective morphism to the Thompson group of piecewise linear automorphisms of $\mathbb{Z}^2$. 
We prove that the presentation for this group conjectured by Usnich is correct.}\end{abstract}
\maketitle

\section{Introduction}
\subsection{The group $\Symp$}
Recall that a rational map $f\colon \C^2\dasharrow \C^2$ -- or a \emph{rational transformation} of $\C^2$ -- is given by $$(x,y)\dasharrow (f_1(x,y),f_2(x,y)),$$ where $f_1,f_2$ are two rational functions (quotients of polynomials) in two variables. The map $f$ is said to be \emph{birational} if it admits a inverse of the same type, which is equivalent to say that $f$ is locally bijective, or that $f$ induces an isomorphism between two open dense subsets of $\C^2$. The group of birational maps of $\C^2$ is the famous \emph{Cremona group}.

Following \cite{bib:Usnich}, we define $\Symp$ as the group of symplectic birational transformations of the plane, which is the group of birational transformations of $\C^2$ which preserve the differential form
$$\omega_0=\frac{dx\wedge dy}{xy}.$$

In \cite{bib:Usnich}, a natural surjective morphism from $\Symp$ to the Thompson group of piecewise linear automorphisms of $\mathbb{Z}^2$ is constructed (see also \cite{bib:Fav}) although the Thompson group does not embedd in the Cremona group. The group $\Symp$, related to other topics of mathematics, is also an interesting subgroup of the Cremona group, from the geometric point of view. The base-points of its elements are poles of the differential form $\omega_0$, but its elements can contract curves which are not poles of $\omega_0$. In this article, we describe the geometry of elements of $\Symp$, and give proofs to two conjectures of \cite{bib:Usnich} (Theorems~\ref{Thm:Generators} and \ref{Thm:relations} below).

\subsection{The results}

The two groups $\SL(2,\ZZ)$ and $(\C^{*})^2$ naturally embedds into $\Symp$; the matrix
$\left(\begin{array}{cc}a & b \\ c & d\end{array}\right)\in \SL(2,\ZZ)$ corresponds to the map $(x,y)\dasharrow (x^ay^b,x^cy^d)$, and the pair $(\alpha,\beta)\in (\C^{*})^2$ corresponds to $(x,y)\dasharrow (\alpha x,\beta y)$. Moreover, the map $P\colon (x,y)\dasharrow (y,\frac{y+1}{x})$, of order $5$, is also an element of $\Symp$. Our first main result consists of proving the following result, conjectured in \cite{bib:Usnich}:
\begin{theo}\label{Thm:Generators}
The group $\Symp$ is generated by $\SL(2,\ZZ)$, $(\C^{*})^2$ and $P$.
\end{theo}
The map $P$ is a well-known linearisable map (\cite{bib:BaB}), and the group $<\SL(2,\ZZ),(\C^{*})^2>$ is a toric well-understood group. The mix of this group with $P$ provides all the complexity to $\Symp$. In the proof, the reader can see that all non-toric base-points come from $P$, but in fact, there are many relations in $\Symp$, and we can have complicated elements with many non-toric base-points.

However, the natural subgroup $H\subset \Symp$ generated by $\SL(2,\ZZ)$ and $P$ is easier to understand. It is an interesting  subgroup of finite type of the Cremona group, which is moreover defined over $\mathbb{Q}$ or over any field. We write $C,I$ the elements
$C=\left(\begin{array}{cc}-1 &  1\\
-1&  0\end{array}\right)$ and $I=\left(\begin{array}{cc}0 &  -1\\
1&  0\end{array}\right)$ of $\SL(2,\ZZ)$. The presentation

$$\SL(2,\ZZ)=<I,C\ |\ I^4=C^3=[C,I^2]=1>$$
is  classical.  We now prove the following result on the relations of $H$, conjectured in \cite{bib:Usnich}:
\begin{theo}\label{Thm:relations}
The following is a presentation of the group $H$:
$$H=<I,C,P\ |\ I^4=C^3=[C,I^2]=P^5=1, PCP=I>.$$
\end{theo}

The author thanks S. Galkin for asking him these questions in the Workshop on the Cremona group organised by I. Cheltsov in Edinburgh in the beginning of $2010$.
%
%
%
%$\begin{array}{ll}
%\frac{d(f_1)\wedge d(f_2)}{f_1f_2}&= (f_1f_2)^{-1}\cdot(\frac{\partial f_1}{\partial x} dx+\frac{\partial f_1}{\partial y} dy)\wedge (\frac{\partial f_2}{\partial x} dx+\frac{\partial f_2}{\partial y} dy)\\
%&= (f_1f_2)^{-1}\cdot(\frac{\partial f_1}{\partial x}\frac{\partial f_2}{\partial y}-\frac{\partial f_1}{\partial y}\frac{\partial f_2}{\partial x})(dx\wedge dy)
%\end{array}$
%
%
%Note that preserving this form amounts to the equality \[\frac{\partial f_1}{\partial x}\frac{\partial f_2}{\partial y}-\frac{\partial f_1}{\partial y}\frac{\partial f_2}{\partial x}=\frac{f_1f_2}{xy}.\]
%
%Note that $\SL(2,\ZZ)$

\section{Some reminders on birational transformations}
Recall that any birational transformation of $\C^2$ extends to an unique birational transformation of the projective complex plane $\p^2$ (written also $\p^2_\C$ or $\C\p^2$) via the embedding $(x,y)\mapsto (x:y:1)$. We will take $X,Y,Z$ as homogeneous coordinates on $\p^2$, so that the afined coordinates $x,y$ on $\C^2$ correspond to $x=X/Z$ and $y=Y/Z$. Any birational transformation $\varphi$ of $\p^2$ can be written as
$$\varphi\colon(X:Y:Z)\dasharrow (P_1(X,Y,Z):P_2(X,Y,Z):P_3(X,Y,Z)),$$
where the $P_i$ are homogeneous polynomials of the same degree without common factor. The degree of the map is the degree of the $P_i$. If this one is $>1$, then there is a finite number of points of  $\p^2$ where $\varphi$ is not defined, which corresponds to the set of common zeros of $P_1,P_2,P_3$. 

More generally, the base-points of $\varphi$ are the points where all curves of the linear system $\sum \lambda_i P_i, \lambda_i\in \C$ pass through. Note that these points are not necessarily on $\p^2$, but maybe in some blow-up, and correspond thus to some tangent directions. See for example \cite{bib:AC} for more details.

\section{Normal cubic forms and geometric descriptions}

Note that the divisor corresponding to $\omega_0$ on $\p^2$ is $-(X)-(Y)-(Z)$.

\begin{defi}We say that a differential form $\omega$ is a \emph{normal cubic form} if $-\div(\omega)$ is the divisor of a $($possibly reducible$)$ singular cubic, whose singular points are ordinary double points. 
\end{defi}

Note that in the above definition, $-\div(\omega)$ can be either (i) the union of three lines with exactly three double points, (ii) the union of a smooth conic and a line intersecting into two distinct points, (iii) an irreducible cubic curve having an unique ordinary double point.
The form $\omega_0$ is a normal cubic form of type (i).

Before using the above definition, we remind the reader the following simple result, already observed in \cite{bib:Usnich}.
\begin{lemm}\label{Lemm:MultZeros}
Let $\omega$ be a differential form on a smooth algebraic surface $S$ and let $\eta\colon \hat{S}\to S$ be the blow-up  of $q\in S$.  We write $D=\div(\omega)$ the divisor of $\omega$, $\tilde{D}$ its strict transform on $\hat{S}$, and $E$ the exceptional curve contracted by $\eta$.

Then $\div(\eta^{*}(\omega))=\tilde{D}+(m+1)E$, where $m\in \mathbb{Z}$ is the multiplicity of $D$ at $q$. In particular, 

\begin{enumerate}
\item
$E$ is a zero of $\div(\eta^{*}(\omega))$ $\Leftrightarrow$ $D$ has multiplicity $\ge 0$ at $q$.
\item
$E$ is a pole of $\div(\eta^{*}(\omega))$ $\Leftrightarrow$ $D$ has multiplicity $\le -2$ at $q$;
\item
$E$ is a pole of multiplicity one of $\div(\eta^{*}(\omega))$ $\Leftrightarrow$ $D$ has multiplicity $-2$ at $q$.
\end{enumerate}
\end{lemm}
\begin{proof}
Let us take some local coordinates $u,v$ on $S$ at $q$ so that this point corresponds to $u=v=0$.
The form $\omega$ locally corresponds to $\varphi(u,v)\cdot du \wedge dv$, where $\varphi$ is a rational function in two variables, and $D$ corresponds to $(\varphi(u,v))$.

The blow-up can be viewed locally as $(u,v)\mapsto (uv,v)$, and  $\eta^{*}(\omega)$  becomes $\varphi(uv,v)\cdot d(uv) \wedge dv=\varphi(uv,v)v\cdot du \wedge dv$. In these coordinates, $v$ is the equation of the divisor $E$ and $\varphi(uv,v)$ corresponds to $\eta^{*}(\div(\omega))$. Moreover $\varphi(uv,v)=v^m \cdot \psi(u,v)$, where $m\in \mathbb{Z}$ is the multiplicity of $D$ at $q$ (which is the multiplicity of $\varphi$ at $(0,0)$), and where $\psi(0,0)\in \C^{*}$. Observing that $\psi(u,v)$ correspons to $\tilde{D}$, we obtain the result. \end{proof}
We can now relate the base points of birational maps to the image of normal cubic forms:

\begin{prop}\label{prop:StdcubicBasePts}
Let $\varphi\colon \p^2\dasharrow \p^2$ be a birational map, and let $\omega$ be a normal cubic form. The following are equivalent:

\begin{enumerate}
\item
All base-points of $\varphi$ are poles of $\omega$;
\item
The form $\varphi_{*}(\omega)$ is a normal cubic form.
\end{enumerate}
\end{prop} 
\begin{proof}
If $\varphi$ has no base-point, both assertions are trivially true, so we may assume that $\varphi$ has at least one base-point.

We denote by $\eta\colon S\to \p^2$ the blow-up of all base-points of $\phi$, and by $\epsilon\colon S\to \p^2$ the morphism $\phi \eta$, which is the blow-up of all base-points of $\phi^{-1}$.

Suppose first that at least one base-point $q$ of $\phi$ (which may be infinitely near to $\p^2$) is not a pole of $\omega$. By Lemma~\ref{Lemm:MultZeros}, the exceptional curve of this point, and of all infinitely near points, are zeros of $\eta^{*}(\omega)$. Since $q$ is a base-point, at least one of these curves is not contracted by $\eta$, and thus $\varphi_{*}(\omega)=\epsilon_{*}(\eta^{*}(\omega))$ has zeros; it is therefore not a normal cubic form.

Suppose now that all base-points of $\varphi$ are poles of $\omega$. If $-\div(\omega)$ is an irreducible cubic curve, it has an unique ordinary double point,  we assume that $\eta$ blows-up this point, by replacing $\eta$ by its composition with the blow-up if needed, obtaining another (non-minimal) resolution of $\varphi$. We now prove the following assertion:
%Then, we prove the following assertion by induction on the number of points blown-up by $\eta$:

{\it The divisor $D_S=-\div(\eta^{*}(\omega))$ is linearly equivalent to $-K_S$ and is an effective reduced divisor consisting of a loop of smooth rational curves $($i.e.\ a finite number of smooth rational curves where each one intersect exactly two others, and  each intersection is transversal$)$.} 

Firstly, since $\div(\omega)$ is linearly equivalent to $K_{\p^2}$, we obtain that $D_S$ is equivalent to $-K_S$ by applying Lemma~\ref{Lemm:MultZeros}. Secondly, we recall that $-\div(\omega)$ is an effective divisor, and that it is either a loop of smooth rational curves or an irreducible nodal cubic curve. In this latter case, writing $\mu\colon \mathbb{F}_1\to \p^2$ the blow-up of the singular point, $-\div(\mu^{*}(\omega))$ is the union of the exceptional curve with the strict transform of the cubic, and is thus a loop of smooth rational curves. We proceed then by induction on the number of points blown-up by $\eta$, applying Lemma~\ref{Lemm:MultZeros} at each step; blowing-up a smooth point on a loop does not change the structure of the loop, and blowing-up a singular point only adds one component. The assertion is now clear.

The fact that $D_S$ is an effective divisor linearly equivalent to $-K_S$ implies that $D=-\div(\varphi_{*}(\omega))=-\div(\epsilon_{*}(\eta^{*}(\omega)))=\epsilon_{*}(D_S)$ is an effective divisor linearly equivalent to $-K_{\p^2}$, and is thus a cubic curve. All components of $D_S$ being rational, $D$ cannot be smooth. It remains to see that all singular points of $D$ are ordinary double points. Writing $\omega'=\varphi_{*}(\omega)$, if $D=-\div(\omega')$ had one other singularities, we can check using  Lemma~\ref{Lemm:MultZeros} that $D_S=-\div(\epsilon^{*}(\omega'))$ would not be a loop.\end{proof}
\section{Decomposition into quadratic maps}
It is well known that any birational transformation of the plane decomposes into quadratics maps. Using Proposition \ref{prop:StdcubicBasePts}, we can deduce the same for elements which send a normal cubic form on another one (Lemma~\ref{Lemm:DecompStup}), and then with a more careful study to elements which preserve the divisor of $\omega_0$ (Proposition~\ref{Prop:SympQuad}).

\begin{lemm} \label{Lemm:DecompStup}
Let $\varphi\colon \p^2\dasharrow \p^2$ be a birational map of degree $d>1$, and let $\omega$ be a normal cubic form. If $\varphi_{*}(\omega)$ is a normal cubic form, there exist quadratic transformations $\phi_1,\dots,\phi_n$ such that

\begin{enumerate}
\item
$\varphi=\phi_n\circ \dots \circ \phi_1$;
\item
for $i=1,\dots,n$, $(\phi_i\circ \dots \circ \phi_1)_{*}(\omega)$ is a normal cubic form.
\end{enumerate}
\end{lemm}
\begin{proof}
We start as in the classical proof of Noether-Castelnuovo theorem, by taking a de Jonqui\`eres transformation $\psi$ (a birational map of $\p^2$ which preserves a pencil of lines) such that each base-point of $\psi$ is a base-point of $\varphi$ and $\varphi\psi^{-1}$ has degree $<d$. The existence of such a $\psi$ can be checked for example in Chapter $8$ of \cite{bib:AC} (see in particular the proof of Theorem 8.3.4). 

Since all base-points of $\varphi$ are poles of $\omega$ (Proposition~\ref{prop:StdcubicBasePts}), the same is true for $\psi$, so $\psi_{*}(\omega)$ is a normal cubic form. 

It remains thus to prove the lemma in the case where $\varphi$ is a de Jonqui\`eres transformation of degree $d>1$, which preserves the pencil of lines passing through $s\in \p^2$. We prove the result by induction on $d$, the case $d=2$ being clear. We follow the classical proof of the theorem of Noether-Castelnuovo.

 The linear system of $\varphi$ (which is the pull-back by $\varphi$ of the system of lines of the plane) consists of curves of degree $d$ passing through $s$ with multiplicity $d-1$ and through $2d-2$ other points $t_1,\dots,t_{2d-2}$ with multiplicity one. 

If at least one of the $t_i$'s is a proper point of $\p^2$, say $t_1$, there exists another  $t_j$,  say $t_2$, and a quadratic de Jonqui\`eres transformation $\phi_1$ with base-points $s,t_1,t_2$. The linear systems of $\phi_1$ and $\varphi$ intersecting into $d-1$ free points, the map $\varphi \circ (\phi_1)^{-1}$ is a de Jonqui\`eres transformation of degree $d-1$. Since $(\phi_1)_{*}(\omega)$ is a normal cubic form, the result follows from the induction hypothesis.

If no one of the $t_i$'s is a proper point of the plane, there exists at least one of these, say $t_1$, which corresponds to a tangent direction of $s$, and another point $t_j$, say $t_2$, which is infinitely near to $t_1$. We choose a proper point $u$ in $\p^2$ which is a pole of $\omega$ and which is not aligned with  $s$ and $t_1$. There exists a  quadratic de Jonqui\`eres transformation $\phi_1$ with base-points $s,t_1,u$. The linear systems of $\phi_1$ and $\varphi$ intersecting into $d$ free points, the map $\theta=\varphi \circ (\phi_1)^{-1}$ is a de Jonqui\`eres transformation of degree $d$. The linear system of $\theta$ is the image by $\phi_1$ of the linear system of $\varphi$; it has one proper base-point distinct from $q$, which corresponds to the "image" of $t_2$ by $\phi_1$ (in the decomposition of $\phi_1$ into blow-ups and blow-downs, the exceptional curve associated to $t_1$ is sent onto two a line of $\p^2$ and $t_2$ is sent onto a  general point of this line). 
 Since $(\phi_1)_{*}(\omega)$ is a normal cubic form, we can apply the preceding case to $\theta$.
\end{proof}

\begin{prop}\label{Prop:SympQuad}
Let $\varphi\colon \p^2\dasharrow \p^2$ be a birational map of degree $>1$, and assume that $$\div(\varphi_{*}(\omega_0))=\div(\omega_0)$$ $($where $\omega_0$ is the differential form $\frac{dx\wedge dy}{xy})$.
Then, there exist quadratic transformations $\phi_1,\dots,\phi_n$ such that

\begin{enumerate}
\item
$\varphi=\phi_n\circ \dots \circ \phi_1$;
\item
for $i=1,\dots,n$, $\div((\phi_i\circ \dots \circ \phi_1)_{*}(\omega_0))=\div(\omega_0)$.
\end{enumerate}
\end{prop}
\begin{rema}
A differential form $\omega$ satisfies $\div(\omega)=\div(\omega_0)$ if and only if $\omega=\mu \omega_0$ for some  $\mu\in\C^*$.  In fact, as we will see after, if a birational map sends $\omega_0$ onto $\omega$, the number $\mu$ is $\pm 1$, and both are possible.
\end{rema}
\begin{proof}
Applying Lemma~\ref{Lemm:DecompStup}, we obtain a decomposition $\varphi=\phi_n\circ \dots \circ \phi_1$ where $\omega_i:=(\phi_i\circ \dots \circ \phi_1)_{*}(\omega_0)$ is a normal cubic form for $i=0,\dots,n$. 

Denote by $m$ the maximal degree of the irreducible components of $-\div(\omega_i)$ for $i=0,\dots,n$, denote by $r$ the minimal index where $\omega_{r}$ has a component of degree $m$. We now prove the result by induction on the pairs $(m,n-r)$, ordered lexicographically.

If $m=1$,  $-\div(\omega_i)$ is the union of three lines for each $i$. Composing the quadratic maps with an automorphism of $\p^2$ which sends $\div(\omega_i)$ onto $\div(\omega_0)$, we can assume that $\div(\omega_i)=\div(\omega_0)$ for each $i$ and obtain the result.

Suppose now that $m=2$, which implies that $0<r<n$, since $\omega_0=\omega_n\not= \omega_r$. The divisor $-\div(\omega_r)$ is the union of a line $L$ and a conic $\Gamma$, and the divisor $-\div(\omega_{r-1})$ is the union of three lines. In particular, the curve $\Gamma_0=(\phi_r^{-1})_{*}(\Gamma)$ is a line and 
$L_0=(\phi_r^{-1})_{*}(L)$ is either a point or a line. This implies that the three base-points $s_1,s_2,s_3$ of $\phi_r^{-1}$ belong to $\Gamma$ (as proper or infinitely near points) and that at least one of them lies on $L$. Up to renumbering, $s_1$ is one of the two points of $\Gamma\cap L$, and $s_2$ is either a proper point of $\Gamma$ or the point infinitely near to $s_1$ corresponding to the tangent of $\Gamma$. The curve $\Gamma_5= (\phi_{r+1})_{*}(\Gamma)$ is either a conic or a line, so at least two of the three base-points $t_1,t_2,t_3$ of $\phi_{r+1}$ belongs to $\Gamma$, and $L_5=  (\phi_{r+1})_{*}(L)$ can be a point, a line or a conic. Up to renumbering, $t_1$ is a proper point of $\Gamma$, and $t_2$ is either another proper point of $\Gamma$, or the point infinitely near to $t_1$ corresponding to the tangent direction of $\Gamma$. We can also assume that if $t_2$ belongs to $\Gamma\cap L$, so does~$t_1$.

\[\xymatrix@R=1mm@C=1cm{
&(L_0,\Gamma_0)&&(L,\Gamma)\ar@{-->}[ll]_{\phi_{r}^{-1}}^{[s_1,s_2,s_3]}
\ar@{-->}[rr]^{\phi_{r+1}}_{[t_1,t_2,t_3]}
&&(L_5,\Gamma_5)
}\]

We now define two proper points $a,b$ of $\Gamma$. If $t_1\in \Gamma\cap L$, the point $b$ is a general point of $\Gamma$ (i.e.\ distinct from the $s_i$ and $t_i$), and otherwise $b$ is such that $L\cap \Gamma=\{s_1,b\}$. The point $a$ is a general point of $\Gamma$ (i.e.\ distinct from $b$ and all $s_i$, $t_i$). We define four birational quadratic maps $\chi_1,\chi_2,\chi_3,\chi_4$ of $\p^2$, with base points $[s_1,s_2,a]$, $[s_1,a,b]$, $[t_1,a,b]$ and $[t_1,t_2,b]$ respectively. We moreover set $\chi_0=\phi_r^{-1}$ and $\chi_5=\phi_{r+1}$. By construction, we have the following: for $i=0,\dots, 4$, $\chi_i$ has its three base-points on $\Gamma$ and at least one of them belongs to $L$, so $\Gamma_i=(\chi_i)_{*}(\Gamma)$ is a line, and $L_i=(\chi_i)_{*}(L)$ is either a point or a line; moreover $\chi_i$ and $\chi_{i+1}$ have two common base-points, so $\theta_i=\chi_{i+1}\circ \chi_{i}^{-1}$ is a quadratic map. We obtain the following commutative diagram:

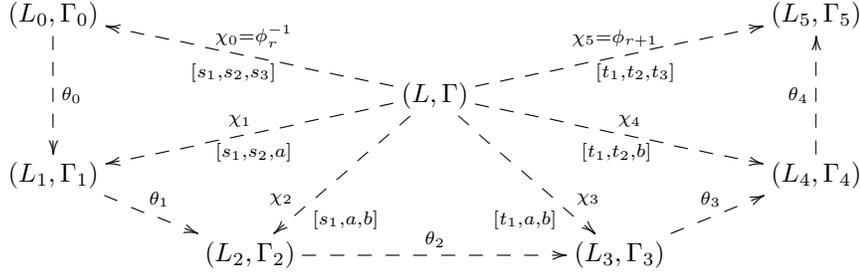
\[\xymatrix@R=0.5cm@C=1.2cm{
(L_0,\Gamma_0)\ar@{-->}[dd]^{\theta_0}&& && (L_5,\Gamma_5)\\
&&(L,\Gamma)
\ar@{-->}[rru]^{\chi_5=\phi_{r+1}}_{[t_1,t_2,t_3]}
\ar@{-->}[llu]_{\chi_0=\phi_{r}^{-1}}^{[s_1,s_2,s_3]}
\ar@{-->}[rrd]^{\chi_4}_{[t_1,t_2,b]}
\ar@{-->}[lld]_{\chi_1}^{[s_1,s_2,a]}
\ar@{-->}[rdd]^>>>>>{\chi_3}_>>>>>>>{[t_1,a,b]}
\ar@{-->}[ldd]_>>>>>{\chi_2}^>>>>>>>{[s_1,a,b]}
&&&\\
 (L_1,\Gamma_1)\ar@{-->}[dr]^{\theta_1} && && (L_4,\Gamma_4)\ar@{-->}[uu]^{\theta_4} \\
& (L_2,\Gamma_2)\ar@{-->}[rr]^{\theta_2}&& (L_3,\Gamma_3)\ar@{-->}[ru]^{\theta_3} }\]

By construction, $-\div((\chi_i)_{*}(\omega_r))$ is the union of three lines for $i=0,\dots,4$; replacing $\phi_{r+1}\circ \phi_r$ by $\theta_4\theta_3\theta_2\theta_1\theta_0$, we reduce the pair $(m,n-r)$.

Suppose now that $m=3$ (which implies that $1<r<n-1$). The divisor $-\div(\omega_r)$ consists of a nodal cubic curve $\Gamma$. The curve $\Gamma_0=(\phi_r^{-1})_{*}(\Gamma)$ is a conic, so all base-points $s_1,s_2,s_3$ of $\phi_r^{-1}$ belong to $\Gamma$, and one of them, say $s_1$, is the singular point of $\Gamma$. Up to reordering, we can assume that $s_2$ is either a proper point of $\Gamma$ or the point infinitely near to $s_1$ corresponding to the tangent of $\Gamma$. We denote by $t_1,t_2,t_3$ the three base-points of $\phi_{r+1}$. The curve $\Gamma_4=(\phi_{r+1})_{*}(\Gamma)$ is either a cubic or a conic, which means that either all $t_i$'s belong to $\Gamma$ or that only two belong to $\Gamma$ but one of these two is the singular point $s_1$.
 If $s_1$ is a base-point of $\phi_{r+1}$, we can assume that $t_1=s_1$, that $t_2$ belongs to $\Gamma$ and that either $t_2$ is a proper point of $\p^2$ or is infinitely near to $t_1=s_1$. If $s_1$ is not equal to any of the $t_i$, we can assume that $t_2$ is a proper point of $\Gamma$.
We choose a general proper point $a$ of $\Gamma$, not collinear with any two of the $s_i,t_i$ and define two birational quadratic maps $\chi_1,\chi_2$ of $\p^2$, with base points $[s_1,s_2,a]$ and $[s_1,t_2,a]$ respectively. We moreover set $\chi_0=\phi_r^{-1}$ and $\chi_3=\phi_{r+1}$. By construction, we have the following: for $i=0,\dots, 2$, $\chi_i$ has its three base-points on $\Gamma$ and at least one of them is $s_1$, so $\Gamma_i=(\chi_i)_{*}(\Gamma)$ is a conic; moreover $\chi_i$ and $\chi_{i+1}$ have two common base-points, so $\theta_i=\chi_{i+1}\circ \chi_{i}^{-1}$ is a quadratic map. We obtain the following commutative diagram:

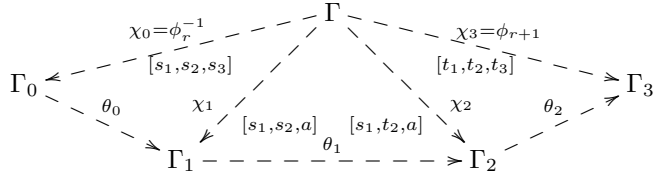
\[\xymatrix@R=0.5cm@C=1.5cm{
&&\Gamma
\ar@{-->}[rrd]^{\chi_3=\phi_{r+1}}_{[t_1,t_2,t_3]}
\ar@{-->}[lld]_{\chi_0=\phi_{r}^{-1}}^{[s_1,s_2,s_3]}
\ar@{-->}[rdd]^>>>>>{\chi_2}_>>>>>>>{[s_1,t_2,a]}
\ar@{-->}[ldd]_>>>>>{\chi_1}^>>>>>>>{[s_1,s_2,a]}
&&&\\
 \Gamma_0\ar@{-->}[dr]^{\theta_0} && && \Gamma_3 \\
& \Gamma_1\ar@{-->}[rr]^{\theta_1}&& \Gamma_2\ar@{-->}[ru]^{\theta_2} }\]

By construction, $-\div((\chi_i)_{*}(\omega_r))$ is the union of the conic $\Gamma_i$ and a line for $i=0,\dots,4$; replacing $\phi_{r+1}\circ \phi_r$ by $\theta_2\theta_1\theta_0$, we reduce the pair $(m,n-r)$.
\end{proof}
\section{Quadratic elements of $\Symp$ and the proof of Theorem~\ref{Thm:Generators}}\label{Sec:Quad}
We now describe some of the main quadratic elements of $\Symp$, useful in the   generation of elements of $\Symp$ (see Proposition~\ref{Prop:SympQuad}).

We fix notation for some points which are poles of $\omega_0$.
The points $p_1,p_2,p_3$ are the vertices of the triangle $XYZ=0$, and $q_1,q_2,q_3$ are points on edges:
 
\[\begin{array}{lll}
p_1=(1:0:0)& p_2=(0:1:0)& p_3=(0:0:1)\\
q_1=(0:1:-1)& q_2=(1:0:-1)& q_3=(1:-1:0)\\
\end{array}\]

Any quadratic birational transformation of $\p^2$ has three base-points. We describe now some quadratic transformations, by giving their description on $\C^2$, $\p^2$ (writing only the image of $(x,y)$ and $(X:Y:Z)$ respectively) and by giving their base-points. Firstly, we describe the classical generators:

\[\begin{array}{rlll}
I^2 & (1/x,1/y) &(YZ:XZ:XY) & p_1,p_2,p_3\\
P & (y,(y+1)/x) & (XY:(Y+Z)Z:XZ) & p_1,p_2,q_1\\
P^2 & ((y+1)/x,(x+y+1)/xy) & (Y(Y+Z):Z(X+Y+Z):XY) & p_1,q_1,q_2\\
P^3 & ((x+y+1)/xy,(x+1)/y) & ((X+Y+Z)Z:X(X+Z):XY) & p_2,q_1,q_2\\
P^4 & ((x+1)/y,x) & (Z(X+Z):XY:YZ) & p_1,p_2,q_2.
\end{array}\]
Secondly, we construct more complicated elements. For any $\lambda\in \C^{*}$, we denote by $\rho_\lambda\in \Aut(\p^2)$ the automorphism $(X:Y:Z)\mapsto (\lambda X:Y:Z)$. If $\lambda\not=-1$, the maps $S_\lambda$ and $T_{\lambda}$, respectively given by
$S_\lambda=(P^2C)^{-1} \rho_{-\lambda} P^2C$ and $T_{\lambda}=P^2\rho_{-\lambda} CP^2$, are described in the following table:

\[\begin{array}{rlll}
S_\lambda& (-\lambda X(X+Y+Z): Y(X+Y-\lambda Z): Z(-\lambda X+Y-\lambda Z))& (0:\lambda:1),q_2,q_3\\
T_{\lambda}&(XY: (Y+Z)(\lambda Z-Y): -\lambda XZ)& p_1,q_1,(0:\lambda:1)
\end{array}\]
Recall that $C$ is the automorphism $(X:Y:Z)\mapsto (Y:Z:X)$ of $\p^2$, which corresponds to the birational map $(x,y)\dasharrow (y/x,1/x)$ of $\C^2$, and thus to the matrix $\left(\begin{array}{cc}-1 &  1\\
-1&  0\end{array}\right)$ of $\SL(2,\ZZ)$. We denote by $\Sym_{(X,Y,Z)}\subset \Aut(\p^2)$ the symmetric group of permutations of the variables, generated by $C$ and $(X:Y:Z)\mapsto (Y:X:Z)$. We now describe linear and quadratic elements of $\Symp$.

\begin{lemm}\label{Lem:SymAut}
The group of automorphisms of $\p^2$ which preserve the triangle $$XYZ=0$$ is $(\C^*)^2\rtimes \Sym_{(X,Y,Z)}$, and its subgroup $(\C^*)^2\rtimes <C>$ is equal to the group of automorphisms of $\p^2$ which are symplectic.
\end{lemm}
\begin{proof}
Follows from a simple calculation.
\end{proof}
\begin{lemm}\label{Lem:QuadProp}

Let $\varphi\colon \p^2\dasharrow \p^2$ be a birational map of degree $2$ which has three proper base-points. The following condition are equivalent:

\begin{enumerate}\item $\div(\varphi_{*}(\omega_0))=\div(\omega_0)$.\item
$\varphi=\alpha Q\beta$, where $\alpha\in (\C^*)^2\rtimes \Sym_{(X,Y,Z)}$, $\beta\in (\C^*)^2\rtimes <C>$ and $Q\in \{I^2,P,P^2,P^3,P^4\}$ or $Q\in\{S_\lambda,T_{\lambda}\}$ for some $\lambda\in \C^{*}\backslash\{-1\}$. \end{enumerate}

\end{lemm}
\begin{proof}
The second assertion clearly implies the first one, since $Q\in \Symp$ is a quadratic map with three proper base-points. It remains thus to prove the other direction.

Denote by $L_1,L_2,L_3\subset \p^2$ the three lines of equation $X=0$, $Y=0$ and $Z=0$. 
Each of the three lines $L_i$ is a pole of  $\omega_0$ and its image by $\varphi$ is thus either a point or a line. So for each $i$, one or two of the base-points of $\varphi$ belong to $L_i$.

Denote by  $k\in \{0,1,2,3\}$ the number of base-points of $\varphi$ which are vertices of the triangle $XYZ=0$. Replacing $\varphi $ by $\varphi C$ or $\varphi C^2$ if neeeded, the $k$ vertices are the $k$ first points of the triple $(p_1,p_2,p_3)$.
  We will find  $Q\in \{I,P,P^2,P^3,P^4,S_\lambda,T_\lambda\}$  and $\beta\in (\C^{*})^2\rtimes \Sym_{(X,Y,Z)}$ such that $\varphi$ and $Q\beta$ have the same base-points.

Before proving the existence of $Q$, $\beta$, let us prove how it yields the result. The fact that $\varphi$ and $Q\beta$ have the same base-points implies that $\varphi=\alpha Q\beta$ for some $\alpha\in \Aut(\p^2)$. Since $\div(\omega_0)=\div(\varphi_{*}(\omega_0))=\div(Q_{*}(\omega_0))=\div(\beta_{*}(\omega_0))$, we also have $\div(\alpha_{*}(\omega_0))=\div(\omega_0)$, which means  that $\alpha\in(\C^*)^2\rtimes \Sym_{(X,Y,Z)}$ (Lemma~\ref{Lem:SymAut}).

We find now $\beta$ and $Q$, by studying the possibilities for $k$.

 If $k=3$, the base-points of $\varphi$ are $p_1,p_2,p_3$ and it suffices to choose $Q=I^2$ and $\beta=1$.
 
 If $k=2$, the base-points are $p_1,p_2,u$, where $u\in (L_1\cup L_2)\backslash L_3$. We choose $\beta\in (\C^{*})^2$ which sends $u$ onto $q_1$ or $q_2$, and choose respectively $Q=P$ or $Q=P^4$.
 
 If $k=1$, the base-points are $p_1,u,v$, where $u\in L_1\backslash (L_2\cup L_3)$. If $v\in L_2$ we choose $\beta \in (\C^{*})^2$ which sends $u$ onto $q_1$ and $v$ onto $q_2$, and choose then $Q=P^2$. If $v\in L_3$, we choose $\beta=\beta'C^{-1}$ , where $\beta'\in (\C^{*})^2$, such that $\beta$ sends respectively $p_1,u,v$ onto $p_2,q_2,q_1$, and choose $Q=P^3$.  If $v\in L_1$, we choose $\beta\in (\C^{*})^2$ which sends $u$ onto $q_1=(0:-1:1)$; the point $v$ is sent onto $(0:\lambda:1)$ for some $\lambda\in \C^{*}\backslash\{-1\}$. We can thus choose $Q=T_\lambda$. 
 
 If $k=0$, the base-points are $u,v,w$, which belong respectively to $L_1$, $L_2$, $L_3$. We choose $\beta\in (\C^{*})^2$ which sends $v$ onto $q_2$ and $w$ onto $q_3$. The point $u$ is sent onto $(0:\lambda:1)$, for some $\lambda\in \C^{*}\backslash\{-1\}$ ($\lambda$ is not $-1$ because $u,v,w$ are not collinear). We choose $Q=S_\lambda$.
 \end{proof}

Now, using all above results, we can prove Theorem \ref{Thm:Generators}, which is a direct consequence of the following proposition.

\begin{prop}\label{Prop:GenCP}
The group $\Symp$ is generated by $(\C^*)^2$, $C$ and  $P$.
\end{prop}
\begin{proof}
Let $f$ be an element of $\Symp$. If its degree is $1$, it is an automorphism of $\p^2$, which is thus generated by $C$ and $P$ (Lemma~\ref{Lem:SymAut}). 

Otherwise, we write $f=\theta_n\circ \dots \circ \theta_1$ using Proposition $\ref{Prop:SympQuad}$, and denote by $m$ the number of $\theta_i$ which have at least one base-point which is not a proper point of $\p^2$. We prove the result by induction on the pairs  $(m,n)$, ordered lexicographically, the case $m=n=0$ being induced by Lemma~\ref{Lem:SymAut}.

Suppose first that the three base-points of $\theta_1$ are proper base-points of $\p^2$. In this case, we apply Lemma~\ref{Lem:QuadProp} and write $\theta_1=\alpha Q \beta$, where $\alpha\in (\C^*)^2\rtimes \Sym_{(X,Y,Z)}$, and $Q$,$\beta$ are generated by $(\C^{*})^2$, $C$ and $P$.
 Replacing $f$ with $f(Q\beta)^{-1}$, we replace the pair $(m,n)$ with $(m,n-1)$.

Suppose now that at least one base-point of $\theta_1$, say $a$, is not a proper point of $\p^2$. 
Denote by $L_1,L_2,L_3\subset \p^2$ the three lines of equation $X=0$, $Y=0$ and $Z=0$. 
Each of the three lines $L_i$ is a pole of  $\omega_0$ and its image by $\theta_1$ is thus either a point or a line. This means that there is at least one base-point on each of the three lines $L_1,L_2,L_3$, and thus that the two other base-points of $\theta_1$ are proper points $b,c\in\p^2$, and at least one of the two points $b,c$ belongs to $L_i$ for $i=1,\dots,3$. We choose some proper point $d$ of the triangle, not aligned with any two of the points $a,b,c$. There exists a quadratic transformation $Q$ of $\p^2$ with base-points $b,c,d$. The map $Q$ sends $\omega_0$ onto a normal cubic form by Proposition~\ref{prop:StdcubicBasePts}. Moreover, the image of any of the lines $L_1,L_2,L_3$ is a line or a point, so $Q$ sends $\omega_0$ onto a normal cubic form corresponding to a triangle. Replacing $Q$ with its composition with an automorphism of $\p^2$, we may assume that the triangle is $XYZ=0$. Because $a$ is not aligned with any two of the points $b,c,d$, the linear system of conics passing through $a,b,c$ is sent by $Q$ onto a system of conics with three proper base-points. In consequence, $\theta_1 Q^{-1}$ is a quadratic transformation with three proper base-points. Replacing $\theta_1$ with $(\theta_1 Q^{-1})\circ Q$, we replace $(m,n)$ with $(m-1,n+1)$.
\end{proof}
\section{The group $H=<\SL(2,\ZZ),P>$}
Let us now focus ourselves on the group $H$ of finite type generated by $\SL(2,\ZZ)$ and $P$, or simply by $C,I,P$ (and in fact only by $P$ and $C$ since $I=PCP$). Recall that $C$ is the automorphism $(X:Y:Z)\mapsto (Y:Z:X)$ of order $3$ of $\p^2$ and that $I$ and $P$ have respectively  order $4$ and $5$. 

Recall the following notation for the points $p_1,p_2,p_3,q_1,q_2,q_3$. 
\[\begin{array}{lll}
p_1=(1:0:0)& p_2=(0:1:0)& p_3=(0:0:1)\\
q_1=(0:1:-1)& q_2=(1:0:-1)& q_3=(1:-1:0)\\
\end{array}\]
We moreover denote by $p_1^{Y}$ the point in the first neighbourhood of $p_1$ which corresponds to the tangent $Y=0$, and do the same for $p_1^Z$, $p_1^{Y+Z}$, $p_2^X$, $p_2^Z$, $q_1^X$ and so on.

We now define twelve quadratic maps contained in $H$, whose three base-points belong to the set $\{p_1,p_1^Y,p_1^Z,p_1^{Y+Z},q_1,q_1^X\}$ or to its orbit by $C$. 
\begin{center}$\begin{array}{rllll}
Q_1=&I & (\frac{1}{y},x) & (Z^2:XY:YZ) & p_1,p_2, p_{1}^Y\\
Q_2=&I^{3} & (y,\frac{1}{x}) &(XY:Z^2:XZ) & p_1,p_2,p_2^{X}\\
Q_3=&I^2 & (\frac{1}{x},\frac{1}{y}) &(YZ:XZ:XY) & p_1,p_2,p_3\\
Q_4=&P & (y,\frac{y+1}{x}) & (XY:(Y+Z)Z:XZ) & p_1,p_2,q_1\\
Q_5=&P^{-1} & (\frac{x+1}{y},x) & (Z(X+Z):XY:YZ) & p_1,p_2,q_2\\
Q_{6}=&PI^2& (\frac{1}{y},x\frac{(y+1)}{y})& (Z^2:(Y+Z)X:YZ)& p_1,p_2, p_1^{Y+Z}\\
Q_{7}=&P^{-1} I^2& (y\frac{(x+1)}{x},\frac{1}{x})& ((X+Z)Y:Z^2:XZ)& p_1,p_2, p_2^{X+Z}\\
Q_8=&I^2 P & (\frac{1}{y},\frac{x}{y+1}) & (Z(Y+Z):XY:Y(Y+Z)) & p_1, p_{1}^Y,q_1\\
Q_9=&I P & (\frac{x}{y+1},y) & (XZ:Y(Y+Z):Z(Y+Z)) & p_1, p_{1}^Z,q_1\\
Q_{10}=&P^2 & (\frac{y+1}{x},\frac{x+y+1}{xy}) & (Y(Y+Z):Z(X+Y+Z):XY) & p_1,q_1,q_2\\
Q_{11}=&P^3C^{-1}& (\frac{x+y+1}{xy},\frac{x+1}{y}) & ((X+Y+Z)Y:Z(Y+Z):XZ) & p_1,q_1,q_3\\
Q_{12}=&PIP&(y,\frac{(y+1)^2}{x})& (XY:(Y+Z)^2:XZ)& p_1,q_1,q_1^X
\end{array}$\end{center}

\bigskip

Any element of $H$ can be written as a word written with the letters $C,I,P$. We will say that a \emph{linear word} is a word of type $C^a$ with $a\in \{0,1,2\}$. Similarly, we will say that a \emph{quadratic word} is a word of type $C^aQ_iC^b$, where 
$1\le i\le 12$, $a,b\in\{0,1,2\}$. Note that a linear word corresponds to a linear automorphism of $\p^2$ and that a quadratic word corresponds to a quadratic birational transformation of $\p^2$.

We would like to prove that the relations $$R=\{ I^4=C^3=[C,I^2]=P^5=1, PCP=I\}$$ (which can easily be verified) generate all the others in $H=<I,C,P>$. To do this (in Proposition~\ref{Prp:Words}), we need to prove some technical simple lemmas (Lemmas~\ref{QuadWinv}, \ref{LemmIiP} and \ref{Lemm:ExistenceQQ}) and one key proposition (Proposition~\ref{fgmKeyProp}).

\begin{lemm}\label{QuadWinv}
If $Q$ is a quadratic word, then $Q^{-1}$ and  $\tau Q\tau^{-1}$ are quadratic words, for any $\tau\in \Sym_{(X,Y,Z)}\subset \Aut(\p^2)$ $($permutation of the coordinates$)$.
\end{lemm}
\begin{proof}
If $\tau=C$, then $\tau Q\tau^{-1}$ is a quadratic word by definition. We can thus assume that $\tau$ is the map $(X:Y:Z)\mapsto (Y:X:Z)$ (or $(x,y)\mapsto (y,x)$), which conjugates $P,I,C$ to respectively $P^{-1},I^{-1},C^{-1}$.  
If $Q$ is a power of $I$ or of $P$, it is clear that $Q^{-1}=\tau Q\tau^{-1}$ is a quadratic word. It remains to study the case when $Q=Q_i$ with $i\in\{6,7,8, 9, 12\}$.

First, we do the case of inverses.
 Since $PCP=I$,  we have $(Q_6)^{-1}=I^2P^{-1}=IPC=Q_9C$, and thus $(Q_9)^{-1}=CQ_6$. Moreover, $(Q_7)^{-1}=I^2P=Q_8$. Finally, using $I^4=1$ and $I=PCP$, we have $$(Q_{12})^{-1}=P^{-1}I^{-1}P^{-1}=P^{-1}(PCP)I(PCP)P^{-1}=CPIPC=CQ_{12}C.$$

Now, the conjugation. We have $\tau Q_6\tau^{-1}=Q_7$ and $\tau Q_8\tau^{-1}=I^2P^{-1}=I(PCP)P^{-1}=IPC=Q_9C$. This implies that $\tau Q_9\tau^{-1}=Q_8 C$. Finally, $\tau Q_{12}\tau^{-1}=(Q_{12})^{-1}$ is a quadratic word, as we just proved.
\end{proof}

\begin{lemm}\label{LemmIiP}
The words
$$ I^aP^{\pm 1},P^{\pm 1}I^a,\ P^{\pm 1}I^aP^{\pm 1},$$
 where $a\in\mathbb{N}$, are equivalent, up to relations in $R$, to quadratic words.
\end{lemm}
\begin{proof}
From the list, we see that any non-trivial power of $I$ or $P$ is a quadratic word. In particular, the case $a=0$ is trivial. Using $PCP=I$, we find the following tabular:
\[\begin{array}{|r|l|l|l|l|l|l|l|l|}
\hline
a&I^aP &  PI^aP  & PI^aP^{-1}\\
\hline
1 & IP=Q_9 &  PIP=Q_{12}  & PIP^{-1}=P^2C\\
2&I^2P=Q_8 & C^{-1}P^{-1}C^{-1}  & PIPC=Q_{12}C\\
3&P^{-1}C^{-1} & PI^{-1}P=C^{-1}  & PI^{-1}P^{-1}=C^{-1}P^3\\
\hline
\end{array},\]the result is now clear for $I^aP$, $PI^aP$ and $PI^aP^{-1}$.

For any $a$, the word $I^aP^{-1}$ is equal to $I^{a-1}(PCP)P^{-1}=I^{a-1}PC$, and is thus also a quadratic word. The words $P^{\pm 1}I^a$ being the inverses of $I^{-a} P^{\pm 1}$, these are also quadratic words (Lemma~\ref{QuadWinv}). The same holds for $P^{-1}I^aP^{-1}=(PI^{-a}P)^{-1}$. It remains to see that $P^{-1}I^aP$ is a quadratic word for each $a$. Since $I^a=PCP I^a P^{-1}C^{-1}P^{-1}$, we find $P^{-1}I^aP=CP I^a P^{-1}C^{-1}$, which is quadratic word since $PI^aP^{-1}$ is one.\end{proof}

\begin{prop}\label{fgmKeyProp}
Let $f$ and $g$ be two quadratic words in $H$. If the two quadratic maps associated have two $($respectively three$)$ common base-points, then $fg^{-1}$ is equal to a quadratic $($respectively linear$)$ word, modulo the relations $R$.
\end{prop}
\begin{proof}
The list of the twelve quadratic words above give the possible base-points of $f$ and $g$: the base-points of $Q_i$ and $CQ_i$ are the same, and the base-points of $Q_iC$ are the image by $C^{-1}$ of the base-points of $h$. 

A quick look at the list shows that if $f$ and $g$ have the same three base-points, then $f=C^{i}g$, for some integer $i$. In particular, $fg^{-1}$ is equal to a linear word.
We have thus only to study the case when exactly two of the three base-points of $f$ and $g$ are common. 

In the sequel, we will use the following observations: $(i)$ we  can exchange the role of $f$ and $g$ since $fg^{-1}$ is a a quadratic word if and only if its inverse $gf^{-1}$ is (Lemma~\ref{QuadWinv}); $(ii)$ we can replace $f$ and $g$ with $C^if$ and $C^jg$ since this only multiplies $fg^{-1}$ by some power of $C$; $(iii)$  we can replace both $f$ and $g$ with their conjugates under any permutation of $(X,Y,Z)$, using Lemma~\ref{QuadWinv}.

Using $(iii)$, we can "rotate" the two common points by acting with $C$, which acts as $p_3\mapsto p_2\mapsto p_1$, $q_3\mapsto q_2\mapsto q_1$, $p_3^Y\mapsto p_2^X\mapsto p_1^Z$ and so on. The possibilities for the two common base-points can thus be reduced to  $\{p_1,p_2\}$, $\{q_1,q_2\}$, $\{q_1,q_1^X\}$ or $\{p_1,u\}$, where $u\in\{q_1,q_2,q_3,p_1^Y,p_1^Z,p_1^{Y+Z}\}$. Conjugating by $(X:Y:Z)\mapsto (X:Z:Y)$ if needed, $u$ may be choosed in $\{q_1,q_2,p_1^Y,p^1_{Y+Z}\}$ only.

We study each case separately.

$(a)$ Case  ${\{p_1,p_2\}}$ -- Using $(ii)$ and reading the list, we can choose that $f,g\in \{P^{\pm 1}, P^{\pm 1}I^2,I^*\}$ (here the star means any power of $I$). If both $f$ and $g$ are powers of $I$, or both are powers of $P$, so is the product $fg^{-1}$, and we are done. If $f$ is a power of $I$ and $g$ is a power of $P$, then $fg^{-1}$ is equal to $I^iP^{\pm 1}$ and the result follows  from Lemma~\ref{LemmIiP}. If $g=P^{\pm 1}I^2$ and $f$ is a power of $I$, then $fg^{-1}$ is again equal to $I^{i}P^{\pm 1}$ for some integer $i$. If $g=P^{\pm 1}I^2$ and $f=P^{\pm 1}$, then $fg^{-1}=P^{\pm 1}I^2P^{\pm 1}$, which is a quadratic word by Lemma~\ref{LemmIiP}.

  $(b)$ Case $\{q_1,q_2\}$ -- The only possibilities for $f,g$ are $P^2$ or $P^3$, and $fg^{-1}=P^{\pm 1}$.
    
  $(c)$ Case $\{q_1,q_1^{X}\}$ -- The only possibilitiy is $f=g=Q_{12}=PIP$, a contradiction.
  
 $(d)$ Case  ${\{p_1,q_1\}}$ -- The third base-point can be respectively $p_2,p_3,q_2,q_3,p_1^Y,p_1^Z$ or $q_1^X$, and this corresponds respectively to 
  $P$, $P^4C^{-1}=I^3P$, $P^2$, $P^3C^{-1}=P^{-1}I^{-1}P$, $I^2P$, $IP$ and $PIP$. In particular, $f$ and $g$ are equal to $f'P$ and $g'P$  where $f',g'\in\{P^{\pm 1},P^{\pm 1}I,I^{\star}\}$. Here, $fg^{-1}$ (or its inverse) belongs to $\{P^{*},I^{*}, P^{\pm 1}I^{*},P^{\pm 1}IP^{\pm 1}\}$ and we are done by Lemma~\ref{LemmIiP}.
  
 $(e)$ Case  ${\{p_1,q_2\}}$ -- The only possibilities for $f,g$ are $P^{-1}$ or $P^2$, and $fg^{-1}=P^{\pm 3}$.

  %$(f)$ Case ${\{p_1,q_3\}}$ -- Here $f,g\in\{PC^{-1}, P^{3}C^{-1}\}$ and $fg^{-1}=P^{\pm 2}$. 
  
  $(f)$ Case ${\{p_1,p_1^Y\}}$ -- Here $f,g\in\{I,I^2P\}$ and $fg^{-1}=(I^2PI^{-1})^{\pm 1}$, a quadratic word by Lemma~\ref{LemmIiP}.
  
  $(g)$ Case ${\{p_1,p_1^{X+Y}\}}$ -- Here $f,g\in\{PI^2,P^{-1}I^2C^{-1}=P^{-1}C^{-1}I^2\}$ and $fg^{-1}=(PCP)^{\pm 1}=I^{\pm 1}$.\end{proof}
  \begin{coro}\label{Coro:SimpliWords}
  Let $W_1,W_2$ be two quadratic words. If $W_2W_1$ corresponds to a birational map of degree $1$ $($respectively $2)$, then $W_2W_1$ is equal, modulo the relations $R$, to a linear word $($respectively to a quadratic word$)$.
  \end{coro}
  \begin{proof}
  The map corresponding to $W_2W_1$ has degree $1$ (respectively $2$) if and only if the maps corresponding to $W_2$ and $(W_1)^{-1}$ have $3$ (respectively $3$) common base-points. The result follows then from Proposition~\ref{fgmKeyProp}.
  \end{proof}
  \begin{lemm}\label{Lemm:ExistenceQQ}
  Let $a_1,a_2,a_3$ be three non-collinear distinct points, such that 
  \begin{itemize}
  \item[$(Q)$]
for $i=1,2,3$, $a_i$ is a base-point of a quadratic word;
  \item[$(P)$]
  for $i=1,2,3$, if $a_i$ is not a proper point of the plane, it is infinitely near to a point $a_j$, $j\not=i$;
  \item[($\diamond$)]
  for any line $L$ of the triangle $XYZ=0$ in $\p^2$, there exists an $a_i$ which belongs to $L$.
\end{itemize}
Then, there exists a quadratic word $Q$ having $a_1,a_2,a_3$ as base-points.
  \end{lemm}
  \begin{proof}
  Let us write  $r=\#\{a_1,a_2,a_3\}\cap \{p_1,p_2,p_3\}\in \{0,1,2,3\}$.
  
   If $r\ge 2$, we can assume that $a_1=p_1$, $a_2=p_2$ (up to renumbering and multiplying by $C$ or $C^2$). The last point $a_3$ being not collinear to $a_1$ and $a_2$, and being a base-point of a quadratic word, it belongs to $\{p_1^Y,p_2^X,p_3,q_1,q_2,p_1^{y+z},p_2^{X+Z}\}$. We can choose $Q=Q_i$ for $i\in \{1,\dots, 7\}$.
   
   If $r=1$, we can assume that $a_1=p_1$. Condition $(\diamond)$ implies that $q_1$ is equal to $a_2$ or $a_3$. The possibilities for the remaining point are $\{p_1^Y,p_1^Z,q_2,q_3,q_1^X\}$, and we can choose $Q=Q_i$ for $i\in\{8,\dots,12\}$.

The case $r=0$ is not possible. Otherwise we would have $$\{a_1,a_2,a_3\}\subset \{q_1,q_2,q_3,q_1^X,q_2^Y,q_3^Z\},$$ which is impossible since $q_1,q_2,q_3$ are collinear.
  \end{proof}

\begin{prop}\label{Prp:Words}
Let $W$ be a word in $I,P,C$. If $W$ corresponds to a birational map of degree $1$ or $2$, it is equal, up to relations $R$, to a linear or quadratic word. In particular, if $W$ corresponds to the identity of $\Bir(\p^2)$, it is equal to $1$ modulo~$R$.
\end{prop}
\begin{proof}If $W$ is a power of $C$, the result is obvious, so we can write  $W=W_k\cdots W_2W_1$ where each $W_i$ is a quadratic word.  Note that many such writings exist. We call $\Lambda_0$ the linear system of lines of $\p^2$. For $i=1,\dots,k$, we denote by $\Lambda_i$ the linear system of $W_i\cdots W_2W_1(\Lambda_0)$ (identifying here the word with the corresponding quadratic map of $\p^2$), and by $d_i$ its degree. Note that $d_k\in\{1,2\}$ is the degree of (the birational map corresponding to) $W$. We write $D=\max\{d_i|\ i=1,\dots,k\}$ and $n=\max\{i\ | d_i=D\}$.

Suppose first that $D=2$. If $k>1$, the map $W_2W_1$ has degree $2$ or $1$ and we can replace it with a single quadratic or linear word (Corollary~\ref{Coro:SimpliWords}).  Continuing in this way, we show that $W$ is equivalent, modulo $R$, to a linear or a quadratic word.

We suppose now that $D>2$, which implies that $1<n<k$. We order the pairs $(D,n)$ using lexicographical order, and proceed by induction. Proving that $(D,n)$ can be decreased, we will reduce to the case $D=2$ studied before.

If $r=\deg(W_{n+1}W_n)\in \{1,2\}$, we can replace $W_{n+1}W_n$ with a single quadratic or linear word (Corollary~\ref{Coro:SimpliWords}), and this decreases $(D,n)$. We can thus assume that $r=\deg(W_{n+1}W_n)\in\{3,4\}$.

We are looking for a quadratic word $Q$ satisfying the following property:

\begin{center}$(\star)$
$\left[\begin{array}{l}
\deg (Q(\Lambda_n))<d_n=\deg(\Lambda_n),\\
\big\{\deg(QW_n),\deg(W_{n+1}Q^{-1})\big\}=\left\{\begin{array}{r} \{2,2\} \mbox{ if } r=\deg(W_{n+1}W_n)=3,\\
\{2,3\} \mbox{ if } r=\deg(W_{n+1}W_n)=4.\end{array}\right.
\end{array}\right.$\end{center}

We first show that such a $Q$ gives us a way to decrease $(D,n)$, before proving that $Q$ exists. 

If $r=3$, both  $QW_n$ and $W_{n+1}Q^{-1}$ have degree $2$ so are equivalent to, up to relations in $R$, to quadratic words $\sigma_1$ and $\sigma_2$ (Corollary~\ref{Coro:SimpliWords}). Replacing $W_{n+1}W_n=(W_{n+1}Q^{-1})(QW_n)$ by $\sigma_2\sigma_1$, we decrease the  pair $(D,n)$. The replacement is described in the following commutative diagram.

\[\begin{array}{|c|}
\hline
\xymatrix@R=0.5cm@C=1cm{
&\Lambda_n
\ar@{.>}[rd]^{W_{n+1}}
\ar@{.>}[dd]_<<<<<<{Q}\\
\Lambda_{n-1}\ar@{.>}[ru]^{W_n}\ar@{->}[dr]^{QW_n}_{\sigma_1} && \Lambda_{n+1} \\
& Q(\Lambda_n)\ar@{->}[ru]^{W_{n+1}Q^{-1}}_{\sigma_2} }
\\
\hline
\deg(QW_n)=\deg(W_{n+1}Q^{-1})=2\\
\hline\end{array}\]

If $r=4$ and  $\deg(QW_n)=2$, $QW_n$ is equivalent to a quadratic word $\sigma_0$. Moreover, since $\deg(W_{n+1}Q^{-1})=3$ and $\deg(Q(\Lambda_n))<D$, we can use the case $r=3$ described before to write $W_{n+1}Q^{-1}$ as a product of two quadratic words $\sigma_2\sigma_1$ satisfying $\deg(\sigma_1Q(\Lambda))<D$. The replacement  of $W_{n+1}W_n$ with $\sigma_2\sigma_1\sigma_0$, described below, decreases the pair $(D,n)$.

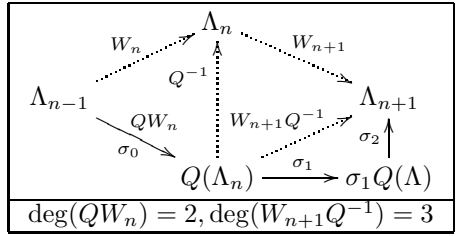
\[\begin{array}{|c|c|}
\hline
\xymatrix@R=0.5cm@C=1cm{
&\Lambda_n
\ar@{.>}[rd]^{W_{n+1}}
\\
\Lambda_{n-1}\ar@{.>}[ru]^{W_n}\ar@{->}[dr]^{QW_n}_{\sigma_0} && \Lambda_{n+1} \\
& Q(\Lambda_n)\ar@{.>}[uu]^<<<<<<<<<<<{Q^{-1}}\ar@{.>}[ru]^{W_{n+1}Q^{-1}}\ar@{->}[r]^{\sigma_1}& \sigma_1Q(\Lambda)\ar@{->}[u]^{\sigma_2} }
\\
\hline
\deg(QW_n)=2,\deg(W_{n+1}Q^{-1})=3 \\ \hline\end{array}\]

If $r=4$ and $\deg(W_{n+1}Q^{-1})=2$, $W_{n+1}Q^{-1}$ is equivalent to a quadratic word $\sigma_0$.
We again apply case $r=3$ (since $\deg(Q(\Lambda_n))<D$) to replace $QW_n$ with a product of two quadratic words $\sigma_2\sigma_1$ with $\deg(\sigma_1(\Lambda_{n-1}))<D$. The replacement of  $W_{n+1}W_n$ with $\sigma_0\sigma_2\sigma_1$, described below, decreases the pair $(D,n)$.
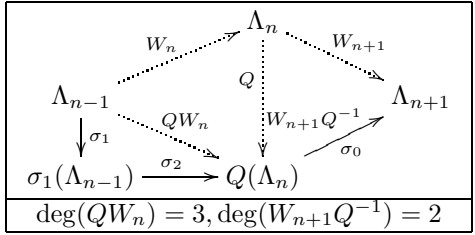
\[\begin{array}{|c|}
\hline
\xymatrix@R=0.5cm@C=1cm{
&\Lambda_n
\ar@{.>}[rd]^{W_{n+1}}
\ar@{.>}[dd]_<<<<<<{Q}\\
\Lambda_{n-1}\ar@{->}[d]^{\sigma_1}
\ar@{.>}[ru]^{W_n}\ar@{.>}[dr]^{QW_n} && \Lambda_{n+1} \\
\sigma_1(\Lambda_{n-1})\ar@{->}[r]^{\sigma_2}
& Q(\Lambda_n)\ar@{->}[ru]^{W_{n+1}Q^{-1}}_{\sigma_0} }
\\
\hline
\deg(QW_n)=3,\deg(W_{n+1}Q^{-1})=2\\ \hline\end{array}\]

It remains to prove the existence of $Q$ satisfying the property $(\star)$. 

We have $D=d_n=\deg(\Lambda_n)$.  The system $\Lambda_{n+1}=W_{n+1}(\Lambda_n)$ has degree $d_{n+1}< D$, and $\Lambda_{n-1}=(W_{n})^{-1}\Lambda_{n}$ has degree $d_{n-1}\le D$. Denote respectively by $S=\{s_1,s_2,s_3\}$ and $T=\{t_1,t_2,t_3\}$ the base-points of $W_{n+1}$ and $(W_{n})^{-1}$. For any point $p$, we will write $m(p)$ the multiplicity of $\Lambda_n$ at $p$.  
The fact that $W_{n+1}$ is a quadratic map with base-points $s_1,s_2,s_3$ implies that $d_{n+1}=\deg(W_{n+1}(\Lambda_n))=2D-\sum_{i=1}^3 m(s_i)$. In particular $\sum_{i=1}^3 m(s_i)> D$. Similarly, $d_{n-1}=2D-\sum_{i=1}^3 m(t_i)$ and $\sum_{i=1}^3 m(t_i)\ge D$.

In order to find $Q$, we will find its three base-points. We are looking for three distinct points 
$a_1,a_2,a_3\in S\cup T$ which satisfy the following conditions: 
\begin{center}\begin{tabular}{ll}
$(\star\star)$&
$\left[\begin{array}{l}
\sum_{i=1}^3 m(a_i)>D,\\
\big\{\{a_1,a_2,a_3\}\cap S,\{a_1,a_2,a_3\}\cap T\big\}=\left\{\begin{array}{r} \{2,2\} \mbox{ if } r=\deg(W_{n+1}W_n)=3,\\
\{1,2\} \mbox{ if } r=\deg(W_{n+1}W_n)=4.\end{array}\right.
\end{array}\right.$\\
$(\diamond)$&
$\left[\mbox{ for each line $L$ of the triangle }XYZ=0\mbox{ one $a_i$ belongs to $L$.}\right.$\end{tabular}\end{center}
The condition $\sum_{i=1}^3 m(a_i)>D$ implies that the three points are not collinear (because $\Lambda_n$ has no fixed component). Replacing a point $a_i$ by $a_i'$ if $a_i$ is infinitely near to $a_i'$ and if $a_i'\notin\{a_1,a_2,a_3\}$, and then applying condition $(\diamond)$, we get a quadratic word $Q$ having $a_1,a_2,a_3$ as its base-points (Lemma~\ref{Lemm:ExistenceQQ}). Condition $(\star\star)$ implies then $(\star)$. 

It remains to find the three points $a_1,a_2,a_3$ satisfying $(\star\star)$ and $(\diamond)$. This is now done separately in the cases $r=3$ and $r=4$.

Suppose that $r=3$, which means that $S\cap T=\{u\}$, for some proper point $u$ of the plane. We order the points of $S$ and $T$ such that $S=\{u,s_1,s_2\}$, $T=\{u,t_1,t_2\}$, with $m(s_1)\ge m(s_2)$ and $m(t_1)\ge m(t_2)$. We observe that at least one of the inequalities $m(u)+m(t_1)+m(s_2)>D$, $m(u)+m(s_1)+m(t_2)>D$ is satisfied. Indeed,  otherwise the sum would give $\sum_{i=1}^3 m(s_i)+\sum_{i=1}^3 m(t_i)\le 2D$, which is impossible. We assume first that $m(u)+m(s_1)+m(t_2)>D$, and write $A_1=\{u,s_1,t_1\}$, $A_2=\{u,s_1,t_2\}$. For $i=1,2$, we have $\sum_{p\in A_i} m(p)\ge m(u)+m(s_1)+m(t_2) >D$, and thus the three points of $A_i$ satisfy condition $(\star\star)$ and in particular are not collinear. We claim now that at least one of the two sets $A_1,A_2$ satisfies condition $(\diamond)$. Suppose the converse for contradiction. This means that for $i=1,2$, there exists a line $L_i$ in the standard triangle $XYZ=0$ such that $L_i\cap A_i=\emptyset$. Since $T=\{u,t_1,t_2\}$ satisfies condition $(\diamond)$, we see that $t_1\in L_2\backslash L_1$ and $t_2\in L_1\backslash L_2$, in particular $L_1\not=L_2$. Denoting by $L_3$ the last line of the triangle, we have $u,s_1\in L_3\backslash( L_1\cup L_2)$.   Since $t_1$ and $t_2$ are not collinear with $u$ and $s_1$, both do not belong to $L_3$. This implies that $T=\{u,t_1,t_2\}=\{q_1,q_2,q_3\}$, which is impossible since $q_1,q_2,q_3$ are collinear (they belong to the line $X+Y+Z=0$). The case $m(u)+m(t_1)+m(s_2)>D$ is the same, by just exchanging $S$ and $T$ in the proof.

Suppose that $r=4$, which means that $S\cap T=\emptyset$. We order the points $s_i$ and $t_i$ such that $m(s_1)\ge m(s_2)\ge m(s_3)$ and  $m(t_1)\ge m(t_2) \ge m(t_3)$. We observe that at least one of the inequalities $m(s_1)+m(t_2)+m(t_3)>D$, $m(t_1)+m(s_2)+m(s_3)>D$ is satisfied. Indeed,  otherwise the sum would give $\sum_{i=1}^3 m(s_i)+\sum_{i=1}^3 m(t_i)\le 2D$, which is impossible. We assume first that $m(s_1)+m(t_2)+m(t_3)>D$, and write $A_1=\{s_1,t_2,t_3\}$, $A_2=\{s_1,t_1,t_3\}$, $A_3=\{s_1,t_1,t_2\}$. For $i=1,2,3$, we have $\sum_{p\in A_i} m(p)\ge m(s_1)+m(t_2)+m(t_3) >D$, and thus the three points of $A_i$ satisfy condition $(\star\star)$. We claim now that at least one of the three sets $A_i$ satisfies condition $(\diamond)$. Suppose the converse for contradiction. This means that for $i=1,2,3$, there exists a line $L_i$ in the standard triangle such that $L_i\cap A_i=\emptyset$. Since $T\cap L_i\not=\emptyset$, we have $t_i\in L_i$ for each $i$ and $t_i\notin L_j$ for $i\not=j$. This implies that the three points $t_i$ are contained in $\{q_1,q_2,q_3,q_1^X,q_2^Y,q_3^Z\}$, which is impossible because $T=\{t_1,t_2,t_3\}$ is the set of base-points of a quadratic word (we can see this on the list of base-points of quadratic words, or simply observe that $q_1,q_2,q_3$ are collinear). The case $m(t_1)+m(s_2)+m(s_3)>D$ is the same, by just exchanging $S$ and $T$ in the proof.
\end{proof}

\end{document}